\documentclass[12pt,letterpaper,oneside]{amsart}
\usepackage[a4paper, left=3cm, right=3cm, bottom=3cm,top=3cm]{geometry}
\usepackage{mathscinet}
\usepackage{amsmath, amsfonts, amssymb, amsthm}
\usepackage{xcolor}
\definecolor{blue}{rgb}{0,0.5,0.5}
\definecolor{ocean}{rgb}{0.00,0.26,0.50}
\usepackage[colorlinks = true, linkcolor=blue, citecolor=ocean]{hyperref}
\newtheorem{theorem}{Theorem}[section]
\newtheorem{proposition}[theorem]{Proposition}
\newtheorem{definition}[theorem]{Definition}
\newtheorem{remark}[theorem]{Remark}

\newtheorem{example}[theorem]{Example}

\newtheorem{corollary}[theorem]{Corollary}
\newtheorem{lemma}[theorem]{Lemma}

\title[Quasiisometry Invariance of Relatively hyperbolic TDLC groups]{Quasiisometry Invariance of Relatively hyperbolic TDLC groups}
\author{Arunava Mandal and Ravi Tomar}

\address{Department of Mathematics,
Indian Institute of Technology Roorkee,
Uttarakhand 247667, India}
\email{arunava@ma.iitr.ac.in}

\address{Beijing International Center for Mathematical Research, Peking University, No. 5 Yiheyuan Road Haidian District, Beijing, P.R.China 100871}
\email{ravitomar547@gmail.com}

\begin{document}

\begin{abstract}
  Motivated by the work of Dru\c{ț}u--Sapir, we introduce a definition of relatively hyperbolic TDLC groups that depends only on their Cayley–Abels graphs and show that this new definition is equivalent to one introduced by Arora--Mart\'inez-Pedroja. Then, we prove that relative hyperbolicity for TDLC groups is invariant under quasiisometry.
\end{abstract}
\maketitle
\subjclass{\it 2020 Mathematics Subject Classification:} Primary 20F65; Secondary 22D05, 20E08

\keywords{\bf Keywords:} Quasiisometries, Asymptotically tree-graded spaces, Relatively hyperbolic TDLC groups.

\setcounter{tocdepth}{2}
\tableofcontents

\section{Introduction}\label{section:intro}

One of the central themes of Gromov's program in geometric group theory is to determine which properties of finitely generated groups are preserved under quasiisometries.
Among various fundamental notions introduced by Gromov \cite{gromov-hypgps} for finitely generated groups is relative hyperbolicity. Since its introduction, relative hyperbolicity has been studied extensively from various perspectives, leading to several equivalent definitions, numerous interesting examples, and deep structural results; see, for instance, \cite{farb-relhyp,bowditch-relhyp,drutu-tree-graded-space,drutu-sapir-rel-hyp-and-RD,osin-book,groves-manning,drutu-qi-rel-hyp, jason-mosher-drutu, hru-rel}. A celebrated result of Druțu \cite{drutu-qi-rel-hyp} shows that relative hyperbolicity is invariant under quasiisometry, and hence is a geometric property. 

In the spirit of Gromov's program, it is natural to ask the analogous question for compactly generated groups. More generally, the large-scale geometry of compactly generated TDLC groups has attracted considerable attention in recent years (see, for example, \cite{baumgartner-moller-willis-flatrank-one, cornulier-qi-classification, cornulier-harpe-book-metric-geometry, arora-castellano-pedroja, arora-pedroja, combination-theorem, locally-quasiconvex-combination-theorem,quasiisometry-tomar-mandal}). The study of relative hyperbolicity in the setting of totally disconnected locally compact (TDLC) groups is comparatively recent. In \cite{caprace-mood-amenable-hyp}, for locally compact groups, the authors studied amenable relatively hyperbolic groups. Later, following Bowditch's approach \cite{bowditch-relhyp}, Arora--Mart\'inez-Pedroja \cite{arora-pedroja} introduced a notion of relative hyperbolicity for compactly generated TDLC groups and established several results in this direction. Subsequently, in \cite{combination-theorem}, the authors introduced two other equivalent definitions of relatively hyperbolic TDLC groups, which are also equivalent to one introduced by Arora--Mart\'inez-Pedroja. Further, they prove combination theorems for relatively hyperbolic TDLC groups and describe their Bowditch boundaries.
Motivated by Dru\c{t}u's work \cite{drutu-qi-rel-hyp} in the finitely generated setting, it is therefore natural to ask whether relative hyperbolicity of compactly generated TDLC groups is preserved under quasiisometries. The following theorem answers this question.

\begin{remark}
    Throughout the paper, for a quasiisometry between compactly generated topological groups, we take word metrics coming from the compact generating sets.
\end{remark}

\begin{theorem}\label{theorem-main}
 Let $G$ and $G'$ be two quasiisometric compactly generated TDLC groups. If $G$ is hyperbolic relative to a family of compactly generated open subgroups $\mathcal H=\{H_1,\cdots, H_n\}$, then $G'$ is hyperbolic relative to compactly generated subgroups $\{H'_1,\cdots, H'_m\}$ and $H_i'$ embeds quasiisometrically in $H_j$ for some $j=j(i)\in\{1,\ldots,n\}.$
\end{theorem}

 We immediately observe the following.
 
\begin{corollary}\label{corollary-main}
 Let $G$ be a compactly generated TDLC group which is hyperbolic relative to a family of compactly generated open subgroups $\{H_1,\ldots, H_n\}$. Then, for a compact TDLC group $K$, the group $G\times K$ is hyperbolic relative to a family of compactly generated open subgroups $\{H'_1,\ldots, H'_m\}$ and $H_i'$ embeds quasiisometrically in $H_j$ for some $j=j(i)\in\{1,2,\ldots,n\}.$  
\end{corollary}
\begin{proof}
    By \cite[Remark 2.17]{combination-theorem}, we see that $G\times K$ is quasiisometric to $G$. Then, the corollary follows from Theorem \ref{theorem-main}.
\end{proof}
In particular, if $D$ is a finitely generated discrete group that is hyperbolic relative to finitely generated 
subgroups $\{H_1,\cdots,H_n\}$, and let $K$ be a compact TDLC group. Then, $D\times K$ is hyperbolic relative to $\{H'_1,\cdots,H'_m\}$, where each $H_j'$ embeds quasiisometrically in some $H_i$. For the definition of amalgamated free products of topological groups and the topology on them, one is referred to \cite{cornulier-harpe-book-metric-geometry}.

\begin{example}
Let $\mathbb Q_p$ be the field of $p$-adic numbers and let $A={\rm SL}(2,\mathbb Q_p)$, and $C$ be a compact open subgroup of $A$. If $\overline{A}$ denotes another copy of $A$, then the amalgamated free product
$G:=A\ast_C \overline{A}$
is hyperbolic relative to $\{A,\overline{A}\}$ \cite{combination-theorem}.
Now, let $B=\operatorname{Aut}(T_{p+1}),$
where $T_{p+1}$ is the $(p+1)$-regular tree. Since both ${\rm SL}(2,\mathbb Q_p)$ and $\operatorname{Aut}(T_{p+1})$ act properly and cocompactly on $T_{p+1}$, then it implies that $A$ and $B$ are quasi-isometric. Set $H:=B\ast_D\bar{B}$. Then, by \cite{quasiisometry-tomar-mandal}, $H$ is quasiisometric to $G$. Although it is known that $H$ is hyperbolic relative to $\{B,\overline{B}\}$, yet Theorem \ref{theorem-main} gives another relatively hyperbolic structure on $H$.
\end{example}
We end the introduction with the following remark.
\begin{remark}
    Let $A,B$ and $C$ be non-compact compactly generated TDLC groups and let $G=A\ast_K B\ast_L C$ be an amalgamated free product, where $K,L$ are compact groups topologically isomorphic to compact open subgroups of $A,B$ and $B,C$, respectively. Then, by \cite[Theorem 1.1]{combination-theorem}, $G$ is hyperbolic relative to $\{A,B,C\}$ as well as hyperbolic relative to $\{A\ast_K B, C\}$. Hence, the exact relation between parabolic subgroups of $G$ is not obvious. However, if we put an extra hypothesis on parabolic subgroups, then we will be able to say the exact relation between parabolic subgroups. More precisely, suppose $G$ is a compactly generated group hyperbolic relative to a collection $\{H_1,\cdots,H_n\}$ of compactly generated open subgroups and each $H_i$ is non-relatively hyperbolic (that is, they are not hyperbolic relative to any collection of non-compact proper subgroups). Then, using the proof of Theorem 4.1 and Theorem 4.8 of \cite{jason-mosher-drutu}, we see that if a compactly generated TDLC group $G'$ is quasiisometric to $G$, then $G'$ is hyperbolic relative to $\{H_1',\cdots,H_m'\}$, where each $H_j$ is quasiisometric to some $H_i$. For example, if $D$ is the mapping class group of a closed orientable surface of genus at least $2$, then $D$ is a non-relatively hyperbolic group. Suppose $K$ is a compact TDLC group and $G$ is a free product of two copies of $D$, i.e. $G=D\ast D$. Then, by Corollary \ref{corollary-main}, $G\times K$ is hyperbolic relative to a collection $\mathcal G$ of subgroups of $G\times K$ such that each subgroup in $\mathcal G$ is quasiisometric to $D$. 
\end{remark}

{\em A few words on the proof:} To prove Theorem \ref{theorem-main}, we introduce a notion of relative hyperbolicity for compactly generated TDLC groups solely in terms of their Cayley-Abel graphs (Definition \ref{definition-tdrh-II}), following Dru\c{t}u--Sapir's approach \cite{drutu-tree-graded-space} in the finitely generated setting. In contrast, the other definitions of relative hyperbolicity in the TDLC framework, including those in \cite{arora-pedroja} and \cite{combination-theorem}, rely not only on a Cayley--Abel graph of the group but also on an additional metric space obtained by attaching suitable geometric objects to the orbits of peripheral subgroups in the Cayley-Abels graph. First, in Theorem \ref{theorem-rh-equivalence}, we show that this notion is equivalent to the existing definitions of relative hyperbolicity for compactly generated TDLC groups. Then, in Theorem \ref{theorem-ATG}, we show that if a Cayley-Abels graph of a compactly generated TDLC group is asymptotically tree-graded with respect to a collection of subsets, then $G$ is either hyperbolic or relatively hyperbolic with respect to a collection of compactly generated open subgroups of $G$. This is the main technical result of this paper. To prove Theorem \ref{theorem-ATG}, we first prove Proposition \ref{main-proposition}, which requires an appropriate modification from the corresponding result in \cite{drutu-qi-rel-hyp}; for example, in the proof, it is required that the peripheral subgroups are compactly generated, and this follows from the authors' previous work in \cite{locally-quasiconvex-combination-theorem}. Finally, we prove Theorem \ref{theorem-main} by using Theorem \ref{theorem-ATG}. 

\vspace{.5cm}

The paper is organized as follows. In Section~\ref{section-preliminaries}, we fix notation, recall a definition of relative hyperbolicity, and provide a brief introduction to Cayley--Abels graphs of compactly generated TDLC groups. In Section~\ref{section-3}, we define a new notion of relative hyperbolicity, prove its well-definedness in Proposition~\ref{prop-well-definedness}, and establish Theorem~\ref{theorem-rh-equivalence}. Finally, in Section~\ref{section-4}, we prove Theorem~\ref{theorem-main} using Proposition~\ref{main-proposition}.

\section{Preliminaries}\label{section-preliminaries}
In this section, we establish notations and recall some basic definitions and results that are relevant to us. In this paper, all topological groups are Hausdorff. Throughout the paper, all graphs are assumed to be connected, and each edge has length one so that graphs are naturally geodesic metric spaces. For a graph $\Gamma$, we denote by $V(\Gamma)$ and $E(\Gamma)$ the set of vertices and edges of $\Gamma$, respectively. For a metric space $(Z,d)$, a subset $Y\subseteq Z$, and $D\geq 0$, we denote the closed $D$-neighborhood of $Y$, i.e. $\{z\in Z: d(z,r)\leq D \text{ for some } r\in R\}$ by $N_D(Y)$. Let $(Y,d_Y)$ and $(Z,d_Z)$ be metric spaces.
Given $\lambda\geq 1,c\geq 0$, a map 
	$f:Y\rightarrow Z$ is said to be a \emph{$(\lambda,c)$-quasiisometric} if, for all $y,y'\in Y$, we have,
	$$\dfrac{1}{\lambda}d_Y(y,y')-c\leq d_Z(f(y),f(y'))\leq \lambda d_Y(y,y')+c.$$

The map $f$ is said to be $(\lambda,c)$-\emph{quasiisometry} if $f$ is a $(\lambda,c)$-quasiisometric embedding and moreover, $N_D(f(Y))=Z$ for some $D\geq 0$.

\begin{definition}
    For $\mathcal K\geq 0$, an action of a group $G$ on a metric space $Z$ is said to be {\em $\mathcal K$-transitive} if ,for $z\in Z$, $Z=N_{\mathcal K}(Gz)$, where $Gz$ denotes the $G$-orbit of $z$.
\end{definition}

{\bf Cayley-Abels graphs:} Given a finitely generated group, one can construct its Cayley graph and thus treat the group as a geometric object. For TDLC groups, an analog construction has been taken into account by Abels \cite{abels}. These are known as {\em Cayley-Abels graphs} named after him. For details on Cayley-Abels graph, see for example \cite[p. 150]{monod-book} and \cite{kron-moller}.

\begin{definition}\label{definition-cayley-abels-graph}
A locally finite connected graph $X$ is said to be a {\em Cayley-Abels graph} of a TDLC group $G$ if $G$ acts transitively on $V(X)$ and stabilizers of vertices are compact open subgroups of $G$.
\end{definition}

 A topological group is said to be {\em compactly generated} if it is algebraically generated by a compact subset. For a TDLC group $G$, a Cayley-Abels graph exists if and only if $G$ is compactly generated \cite[Theorem 2.2]{kron-moller}.

 \vspace{.2cm}

 {\bf Existence of a Cayley-Abels graph:} Let $G$ be a compactly generated TDLC group, and $U$ be a compact open subgroup of $G$. If $K$ is a compact generating set of $G$, there exists a finite symmetric set $S$ containing the identity element of $G$ such that $K\subset SU$. Define a graph $X(K,S,U)$ whose vertex set is the set of left cosets of $U$ in $G$, and the edge set is $\{\{gU,gsU\}:g\in G \text{ and } s\in S\}$. Then, it is easy to check that $X$ is a Cayley-Abels graph for $G$ (see  \cite{kron-moller}). We write $X$ in place of $X(K,S,U)$ when $K,S,U$ are clear from the context.

Let $\mathcal H=\{H_1,\dots, H_n\}$ be a finite collection of compactly generated open subgroups of $G$. Let $U$ be a compact open subgroup of $G$. Let $K$ be a compact generating set of $G$, and let $K_i$ be a compact generating set of $H_i$ for $1\le i\le n$. Note that $U\cap H_i$ is an infinite compact open subgroup of $H_i$ for each $i$ by Lemma 2.4 of \cite{combination-theorem}. Then, there exists a finite symmetric set $S\subset G$ containing the identity element such that $K\subset SU$. Also, for each $i$, there is a finite symmetric set $S_i\subset H_i$ containing identity such that $K_i\subset S_i(U\cap H_i)$. We assume that $S_i\subset S$ for all $i$. Thus, the Cayley-Abels graph $Y_i=Y_i(K_i,S_i,U\cap H_i)$ of $H_i$ is embedded in the Cayley-Abels graph $X=X(K,S,U)$ of $G$. 
Therefore, we can treat $Y_i$ as a subgraph of $X$.
\begin{lemma}\label{lemma-bijection}
    There is a bijection between $V(Y_i)$ and the $H_i$-orbit $\mathcal O_{H_i}$ of $U$ in $X$.
\end{lemma}
\begin{proof}
    Define a map $\phi:V(Y_i)\to \mathcal O_H$ which take $h(U\cap H_i)$ to $hU$. Suppose $h(U\cap H_i)=h'(U\cap H_i)$. Then, $h^{-1}h'\in U\cap H_i$ and hence $h^{-1}h'\in U$. This implies that $hU=h'U$ and thus $\phi$ is well-defined. Surjectivity of $\phi$ is clear. Suppose $\phi(h(U\cap H_i))=\phi(h'(U\cap H_i))$. Then, $hU=h'U$ and therefore $h^{-1}h'\in U\cap H_i$. This implies that $h(U\cap H_i)=h'(U\cap H_i)$ and hence $\phi$ is injective. This completes the proof of the lemma.
\end{proof}

For each $i\in\{1,2,\dots,n\}$, let $T_i$ be a left transversal for $H_i$ in $G$. For each $i$ and each $t\in T_i$, let $\mathcal O_{i,t}$ denote the left translate of $\mathcal O_{H_i}$ by $t$. 
Note that by Lemma \ref{lemma-bijection}, each $\mathcal O_{i,t}$ is bijective to the vertex set of $Y_i$.

\begin{definition}[Coned-off Cayley-Abels graph] \label{definition-coned-off-graph}
Form a new graph $\hat{X}=\hat{X}(K,S,U)$, called the {\em coned-off Cayley-Abels graph with respect to $\mathcal H$}, as follows. For each $\mathcal O_{i,t}$, add a new vertex, called the {\em cone point}, $v(\mathcal O_{i,t})$ to $X$ and add an edge of length $1$ from this new vertex to each element of $\mathcal O_{i,t}$.
\end{definition}


The pair $(G,\mathcal H)$ is said to be a {\it proper pair} if no pair of distinct non compact subgroups in $\mathcal H$ are conjugate in $G$ (page-832, \cite{arora-pedroja}). From now on, we assume that $(G,\mathcal H)$ is a proper pair.

\begin{definition}[Relatively hyperbolic TDLC groups]\label{definition-tdrh-I}
 The group $G$ is said to be {\em hyperbolic relative to $\mathcal H$} if $\hat{X}$ is a hyperbolic graph and it satisfies the bounded penetration property (BPP).
\end{definition}

For the definition of BPP, one is referred to \cite{farb-relhyp}. 

\begin{lemma}\label{lemma-hyp-implies-rel-hyp}
  If $G$ is a compactly generated hyperbolic TDLC group, then $G$ is hyperbolic relative to any compact open subgroup.  
\end{lemma}

\begin{proof}
    Let $U$ be a compact open subgroup of $G$. Consider the Cayley-Abels graph $X$ of $G$ using $U$. Then, it is straightforward to check that $G$ is hyperbolic relative to $\{U\}$.
\end{proof}

In the next section, we introduce a new definition of relatively hyperbolic TDLC group in terms of asymptotic cones and show that it is equivalent to Definition \ref{definition-tdrh-I} (Theorem \ref{theorem-rh-equivalence}). We end this section with the following remark.
\begin{remark}
    In \cite[Definition 4.7]{combination-theorem}, the authors take the subgraphs $Y_i$'s instead of their vertex sets. By Lemma \ref{lemma-bijection}, we have observe that $V(Y_i)$ is bijective to $\mathcal O_{H_i}$. Since the Hausdorff distance between $V(Y_i)$ and $Y_i$ is bounded above by $1$, Definition \ref{definition-tdrh-I} is equivalent to \cite[Definition 4.10]{combination-theorem} and hence, by \cite[Theorem 4.12]{combination-theorem}, Definition \ref{definition-tdrh-I} is equivalent Definition 2.10 and Definition 4.5 of \cite{combination-theorem}.
\end{remark}
\section{Relatively hyperbolic TDLC groups and asymptotically tree-graded spaces}\label{section-3}
The notion of tree-graded spaces was introduced by Drutu--Sapir \cite{drutu-tree-graded-space}. In the same paper, they introduced the notion of asymptotically tree-graded metric spaces. We recall from \cite{drutu-tree-graded-space}, a metric space $Z$ is asymptotically tree-graded (ATG) with respect to a collection of subsets $\mathcal A$ if every asymptotic cone of $Z$ is tree -graded with respect to a collection of limit sets of sequences in $\mathcal A$. For the definition of asymptotic cone, we refer to Section 3 of \cite{drutu-tree-graded-space}.
Equivalently, by \cite[Theorem 4.1]{drutu-tree-graded-space}, $Z$ is asymptotically tree-graded with respect to $\mathcal{A}$ if the following geometric properties are satisfied:

\begin{enumerate}
\item[$(\alpha_1)$] For every $\delta>0$, the intersections of the $\delta$-neighborhoods of distinct elements of $\mathcal{A}$ have uniformly bounded diameter.

\item[$(\alpha_2)$] There exists $M>0$ such that every geodesic whose endpoints lie at distance at most one third of its length from some $A\in\mathcal{A}$ intersects the $M$-neighborhood of $A$.

\item[$(\alpha_3)$] There exists $R>0$ such that every fat geodesic polygon is contained in the $R$-neighborhood of some set $A\in\mathcal{A}$ (see \cite[Definition 3.32]{drutu-tree-graded-space} for fat polygon).
\end{enumerate}

The space $Z$ is said to {\em properly} ATG with respect to $\mathcal A$ if $Z$ is not contained in $N_r(A)$ for any $A\in\mathcal A$ and $r\geq 0$. From now on, we assume that all ATG metric spaces are properly ATG.

For ease of reference, we also include the following properties of ATG spaces. Suppose a metric space $Z$ is ATG with respect to $\mathcal A$.

\begin{enumerate}
\item [$(\beta_2)$] There exists $\epsilon>0$ and $M\geq 0$ such that any geodesic $\alpha$ of length $l$ and $A\in\mathcal A$ satisfying $\alpha(0),\alpha(l)\in N_{\epsilon l}(A)$, the middle third $\alpha([l/3,2l/3])$ is contained in $N_M(A)$.

\item [$(\beta_3)$] There exists $\theta>0,\nu\geq 8$ and $\mathcal X>0$ such that any geodesic hexagon $(\theta,\nu)-$ fat is contained in $N_{\mathcal X}(A)$ for some $A\in\mathcal A$.

\item [(Qconv)] There exists $t> 0$ and $K_0\geq 0$ such that for every $A\in \mathcal A$, $K\geq K_0$ and $x,y\in N_K(A)$, every geodesic joining $x$ and $y$ in $Z$ is contained in $N_{tK}(A)$.
\end{enumerate}

Motivated by the work of Drutu--Sapir \cite{drutu-tree-graded-space} for discrete relatively hyperbolic groups, we introduced the following definition in the realm of compactly generated TDLC groups. We continue to use the same notation from the previous section.

\begin{definition}\label{definition-tdrh-II}
 We say that $G$ is hyperbolic relative to $\mathcal H$ if a Cayley-Abels graph $X$ is asymptotically tree-graded with respect to the collection $\mathcal L=\{\mathcal O_{i,t}: t\in T_i \text{ and } 1\leq i\leq n\}$.
\end{definition}
The following proposition shows the well-definedness of the above definition.
\begin{proposition}\label{prop-well-definedness}
Definition \ref{definition-tdrh-II} is well-defined.
\end{proposition}
\begin{proof}
Suppose $X=X(K,S,U)$ and $X'=X'(K',S',U')$ are two Cayley-Abels graphs of $G$, where $K$ and $K'$ are compact generating sets of $G$, $U$ and $U'$ are compact open subgroups of $G$, and $S,S'$ are finite symmetric generating sets containing the identity such that $K\subset SU$ and $K'\subset S'U'$. For each $i$ and each $t\in T_i$, let $\mathcal O_{i,t}'$ denote the left translate of $\mathcal O'_{H_i}$ by $t$, where $\mathcal O'_{H_i}$ denotes the $H_i$-orbit of $U'$ in $X'$. 

{\bf Claim:} If $X$ is asymptotically tree-graded with respect to the collection $\mathcal L$ then $X'$ is asymptotically tree-graded with respect to the collection $\mathcal L'=\{\mathcal O'_{i,t}: t\in T_i \text{ and } 1\leq i\leq n\}$.

By \cite[Theorem 2.7$^{+}$]{kron-moller}, there exists a quasiisometry $\psi:X\to X'$ such that $\psi(\mathcal O_{i,t})\subset N_D(\mathcal O_{i,t}')$ for a uniform constant $D\geq 0.$ By \cite[Theorem 5.1(3)]{drutu-tree-graded-space}, $X'$ is asymptotically tree-graded with respect to $\psi(\mathcal O_{i,t})$. Finally, by \cite[Remark 4.2(1)]{drutu-qi-rel-hyp}, we are done. This completes the proof of the claim and hence the proposition.
\end{proof}
By \cite[Theorem 1.1]{sisto-metric-relative-hyperbolicity}, the following is immediate.

\begin{theorem}\label{theorem-rh-equivalence}
Definition \ref{definition-tdrh-I} and Definition \ref{definition-tdrh-II} are equivalent.                               \qed
\end{theorem}


\section{Proof of the main result}\label{section-4}

The goal of this section is to give a proof of Theorem \ref{theorem-main}. In that direction, the following result is a key result whose statement and the idea of the proof are borrowed from \cite[Proposition 5.1]{drutu-qi-rel-hyp} with appropriate modifications.

\begin{proposition}[Equivariant ATG structure implies relative hyperbolicity]\label{main-proposition}
Suppose $(X,d)$ is ATG with respect to a collection of subsets $\mathcal B$ of $V(X)$. Assume that $G$ permutes the subsets in $\mathcal B$. Then, we have the following:
\begin{enumerate}
\item There are only finitely many subsets in $\mathcal B$ which contain $U\in V(X)$.

\item  Let $\mathcal F=\{B_1,\cdots,B_k\}$ be the set of $B\in\mathcal B$ containing $U\in V(X)$. For every $B\in\mathcal B$, the stabilizer ${\rm Stab}_G(B)=\{g\in G:gB=B\}$ acts $\mathcal K$-transitively on $B$ for some $\mathcal K\geq 0$. Moreover, for $i=1,2,\cdots,k$, $\operatorname{Stab}_G(B_i)U\subseteq B_i\subseteq N_{\mathcal K}(\operatorname{Stab}_G(B_i)U),$ where $\operatorname{Stab}_G(B_i)U$ denotes the orbit of $U$.

\item Let $\mathcal K$ be the constant as in $(2)$ and let $D_{2\mathcal K}$ be the uniform bound given by property $(\alpha_1)$ for $(X,\mathcal B)$. If all the subsets in $\mathcal B$ have diameter at most $D_{2\mathcal K}+1$, then $G$ is hyperbolic.

\item Let $\mathcal B'$ be the set of $B\in\mathcal B$ of diameter larger than $D_{2\mathcal K}+1$ and let $\mathcal F'=\mathcal F\cap \mathcal B'.$ Then there exists a subset $\mathcal F_0$ of $\mathcal F'$ such that for every $B\in\mathcal B'$, the intersection of the $G$-orbit of $B$ and $\mathcal F_0$ is singleton.

\item $X$ is ATG with respect to $\mathcal B'.$

\item Let $\mathcal F_0=\{\bar{B_1},\ldots,\bar{B_m}\}$ and $H_j={\rm Stab}_G(\bar{B_j})$. For every $B\in\mathcal B'$ there exists a unique $j\in\{1,\cdots,m\}$ and unique left coset $gH_j$ such that $gH_jU\subset B\subset N_{2\mathcal K}(gH_jU)$, where $\mathcal K$ is the constant as in $(2)$ and $H_jU$ denotes the $H_j$-orbit of $U$ in $X$.

\item For $1\le j\le m$, $H_j$ is compactly generated open subgroup of $G$ and $G$ is hyperbolic relative to $\{H_1,\cdots,H_m\}.$ Moreover, for each $H_j$ there exists a unique $B_i\in\mathcal B$ such that $H_jU\subset B_i\subset N_{2\mathcal K}(H_jU)$, where $\mathcal K$ is the constant as in $(2).$ 
\end{enumerate}
    
\end{proposition}
\begin{proof}
   $(1)$ By the quasiconvexity property $(\mathrm{Qconv})$ of $\mathcal{B}$, there exists a constant $\sigma>0$ such that, whenever $x,y\in B$ for  $B\in\mathcal{B}$, every geodesic segment in $X$ joining $x$ and $y$ is contained in $N_{\sigma}(B)$. Moreover, since $\mathcal B$ satisfies property $(\alpha_1)$, there exists a constant $D_{\sigma}>0$ such that, for any two distinct sets $B,B'\in\mathcal{B}$,
$\operatorname{diam}\bigl(N_{\sigma}(B)\cap N_{\sigma}(B')\bigr)\le D_{\sigma}.$ We consider the following two cases.

\medskip
\noindent\textbf{Case 1:} $U\in B$ and $\operatorname{diam}(B)\le 3D_{\sigma}$.

Note that every point of $B$ lies at a distance at most $3D_{\sigma}$ from $U$. Hence, we have $B\subseteq B(U,3D_{\sigma}),$
where $B(U,3D_{\sigma})$ denotes the closed ball of radius $3D_{\sigma}$ centered at $U$.
Since $B(U,3D_{\sigma})$ is finite, there are only finitely many possibilities for such subsets $B$ in $\mathcal B$.

\medskip

\noindent\textbf{Case 2:} $U\in B$ and $\operatorname{diam}(B)>3D_{\sigma}$.

For each such $B$, there exists a point $xU\in B$ such that $d(U,xU)>3D_{\sigma}.$
By quasiconvexity of $B$, geodesic segments joining $U$ and $xU$ is contained in $N_{\sigma}(B)$. Since $d(U,xU)>3D_{\sigma}$, this geodesic intersects the sphere
$S(U,2D_{\sigma})=\{yU\in X\mid d(U,yU)=2D_{\sigma}\}.$
Therefore, $N_{\sigma}(B)\cap S(U,2D_{\sigma})\neq\varnothing.$
Define a map
\[
\Phi:\{B\in\mathcal{B}\mid U\in B,\ \operatorname{diam}(B)>3D_{\sigma}\}
\longrightarrow \mathcal{P}(S(U,2D_{\sigma}))
\]
by
$\Phi(B)=N_{\sigma}(B)\cap S(U,2D_{\sigma}),$ where $\mathcal{P}(S(U,2D_{\sigma}))$ denotes the power sets.
We claim that $\Phi$ is injective. 
By property $(\alpha_1)$ and the choice of $D_{\sigma}$, we see that the sets $N_{\sigma}(B)\cap S(U,2D_{\sigma})$ and $N_{\sigma}(B')\cap S(U,2D_{\sigma})$
are disjoint whenever $B\neq B'$. Hence $\Phi$ is injective.
Since the number of elements in $S(U,2D_{\sigma})$ is finite, its power set $\mathcal{P}(S(U,2D_{\sigma}))$ is also finite. Therefore, only finitely many subsets $B\in\mathcal{B}$ can occur in this second case.
Combining the two cases, we conclude that only finitely many subsets in $\mathcal{B}$ contain $U$.


(2) For each $i=1,\ldots,k$, let $\mathcal I_i=\{j\in\{1,2,\ldots,k\}:gB_i=B_j\; \text{for some}\; g\in G\}$. For every $j\in\mathcal I_i$, let $g_j\in G$ be such that $g_jB_i=B_j$. Set $\mathcal K_i={\rm max}_{j\in \mathcal I_i}d(U,g_jU)$ and set $\mathcal K={\rm max}_{1\leq i\leq k}\mathcal K_i$.
   We show that, for every $B\in\mathcal B$, ${\rm Stab}_G(B)$ acts $\mathcal K$-transitively on $B$.

Let $xU, bU\in B$ be arbitrary elements. Clearly, $b^{-1}B$ and $x^{-1}B$ contain $U$ and are in $\mathcal{B}$. Thus, 
$b^{-1}B = B_i$ and $x^{-1}B = B_j$
for some $i,j \in \{1,2,\ldots,k\}$.
Since $b^{-1}xB_j = B_i,$ it implies that $B_j = g_j B_i$, where $g_j=x^{-1}b$, and hence $x^{-1}B = g_j b^{-1}B.$ This implies that
$xg_j b^{-1} \in \operatorname{Stab}_G(B).$ Note that, $d(xU,\operatorname{Stab}_G(B)bU)\le d(xU,xg_jb^{-1}bU)=d(U,g_jU)\le \mathcal K$.  Now, the moreover part of the statement is clear.

(3) If the diameter of elements in $\mathcal B$ is at most $D_{2\mathcal K}+1$, then the hyperbolicity of $G$ follows from \cite[Corollary 4.21]{drutu-qi-rel-hyp}.

(4) The set $\mathcal F_0$ can be obtained by considering one by one elements $B_i\in\mathcal F'$ and deleting from $\mathcal F'$ all $B_j$ such that $j\in\mathcal I_i$ and $j\neq i$.

(5) It is clear that $G$ permutes the subsets in $\mathcal B'$. By \cite[Corollary 4.26]{drutu-qi-rel-hyp}, $X$ is ATG with respect to $\mathcal B'$.

(6) {\em Existence.} Let $B\in\mathcal B$. Then, by (4), there exists a unique $j\in\{1,\cdots,m\}$ such that $hB=\bar{B}_j$, where $h\in G$. Setting $g=h^{-1}$, we see that $B=g\bar{B}_j.$ By the moreover part of (2), we have that $gH_jU\subseteq B\subseteq N_{\mathcal K}(gH_jU).$

{\em Uniqueness.} Suppose, for $j,l\in\{1,\cdots,m\}$, $gH_jU$ and $g'H_lU$ both satisfy the last inequality of the previous paragraph. Then, $$g\bar{B}_j\subset N_{\mathcal K}(gH_jU)\subset N_{\mathcal K}(B)\subset N_{2\mathcal K}(g'H_lU)\subset N_{2\mathcal K}(g'\bar{B}_l).$$ Note that $g\bar{B}_j$ and $g'\bar{B}_l$ are in $\mathcal B'$. This implies that the diameter of $g\bar{B}_j$ is at least $D_{2\mathcal K}+1$. By Property $(\alpha_1)$, we see that $g\bar{B}_j=g'\bar{B}_l$. The definition of $\mathcal F_0$ implies that $j=l$. Then, $g^{-1}g'\in H_j$ and hence $gH_j=g'H_l$.

(7) By (5), (6) and \cite[Remark 4.2]{drutu-qi-rel-hyp}, we see that $X$ is ATG with respect to $\{gH_jU:g\in G/H_j, 1\le j\le m\}$. This implies that each $H_jU$ is quasiconvex in $X$. Since each $H_j$ contains $U$ that is an open subgroup of $G$, $H$ is an open subgroup of $G$. By \cite[Lemma 3.3(1)]{locally-quasiconvex-combination-theorem}, each $H_j$ is compactly generated. Finally, by Definition \ref{definition-tdrh-II}, $G$ is hyperbolic relative to $\{H_1,\cdots,H_m\}.$ The moreover part is now clear from (6). If $G=H_j$ then the moreover part of $(2)$ implies that $V(X)=\bar{B}_j$. This completes the proof of the proposition.
\end{proof}

The following theorem shows that if a Cayley-Abels graph $X$ of a compactly generated TDLC group $G$ is ATG with respect to a collection of subsets of $V(X)$, then $G$ is either hyperbolic or relatively hyperbolic with respect to a collection of compactly generated open subgroups.
\begin{theorem}\label{theorem-ATG}
 Suppose that $(X,d)$ is ATG with respect to a collection of subsets $\mathcal A$ of $V(X)$. 
 Then $G$ is either hyperbolic or relatively hyperbolic with respect to a family of subgroups $\{H_1,H_2,\ldots,H_m\}$ such that, for each $j$, the $H_j$-orbit of $U$ is contained in $N_L(A)$ for some $A\in\mathcal A$, where $L$ is a constant depending only on $(X,d)$ and $\mathcal A$.  
\end{theorem}
\begin{proof}
    In view of Proposition \ref{main-proposition}, it is sufficient to construct a collection $\mathcal B$ of subsets of $V(X)$ such that $\mathcal B$ is $G$-equivariant and $X$ is ATG with respect to $\mathcal B$. One can prove it by following the proof \cite[Theorem 5.4]{drutu-qi-rel-hyp}. Hence, we skip its proof.
\end{proof}

Now, we are ready to give a proof of our main theorem.
\vspace{.2cm}

{\em Proof of Theorem \ref{theorem-main}:} When $G$ is compact, then its Cayley-Abels graph is finite. By \cite[Lemma 2.7]{quasiisometry-tomar-mandal}, $G$ is quasiisometric to its Cayley-Abels graph. This implies that $G'$ is quasiisometric to a finite graph and thus any Cayley-Abels graph of $G'$ must be finite. This, in turn, implies that $G'$ is compact. We assume henceforth that both $G$ and $G'$ are non-compact. 

Let $X=X(K,S,U)$ and $X'=X'(K',S',U')$ be Cayley-Abels graphs of $G$ and $G'$, respectively. Suppose $q:X\to X'$ is a $(\lambda,c)$-quasiisometry, and $\bar{q}:X'\to X$ is its quasiinverse so that $d(q\circ \bar{q}, id_{X'})\leq D$ and $d(\bar{q}\circ q,id_X)\leq D$, where $D=D(\lambda,c)$. Since $G$ is hyperbolic relative to $\mathcal H$, $X$ is ATG with respect to $\mathcal A:=\{gH_iU: g\in G/H_i \text{ and } 1\le i\le n\}.$ By \cite[Theorem 5.1]{drutu-tree-graded-space}, $X'$ is ATG with respect to $\{q(A):A\in\mathcal A\}$. Moreover, all the constants appearing in the properties $\alpha_i$, $i=1,2,3$, $\beta_j$, $j=2,3$, and (Qconv) for $(X',q(\mathcal A))$ can be formulated as functions of $(\lambda,c)$ and of the constants in the similar properties for $(X,\mathcal A)$.

By Theorem \ref{theorem-ATG}, we see that $G'$ is either hyperbolic or $G'$ is hyperbolic relative to $\mathcal H'=\{H_1',\cdots,H_m'\}$, and each $H_i'U'$ is contained in $N_K(q(A_j))$ for some $A_j\in\mathcal A$, where $K$ is a uniform constant depending only on $(\lambda,c)$ and the constant in the properties $(\beta_2)$ and $(\beta_3)$ for $(X,\mathcal A)$. Let $\pi_1:N_K(q(A_i))\to q(A_i)$ be a map such that $d(x,\pi_1(x))\le K$. Then, it is easy to check that $\pi_1$ is $(1,2K)$-quasiisometric embedding. Similarly, let $\pi_2:N_D(A_i)\to A_i$ be a $(1,2D)$-quasiisometric embedding. Now, the restriction of $\pi_2\circ \bar{q}\circ \pi_1$ to $H_i'U'$ is a $(\lambda',c')$-quasiisometric embedding into $gH_jU$ for some $j\in\{1,\cdots,n\}$ and with $(\lambda',c')$ depending only on $(\lambda,c),K,D$. Hence, each $H_i'$ embeds quasiisometrically into some $H_j$. 

If $G'$ is hyperbolic, then, by Lemma \ref{lemma-hyp-implies-rel-hyp}, $G'$ is hyperbolic relative to any compact open subgroup $U\neq G$. Thus, all the statements in the theorem hold. If $G'=H_i'$ then $X'=N_K(q(A_j))$, which implies that $X\subset N_c(\bar{q}(G'))\subset N_{\lambda K+2c+D}(A_j)$. This is a contradiction, as $X$ is a proper ATG with respect to $\mathcal A$.      \qed

\vspace{.2cm}
\noindent

\vspace{.2cm}
\noindent
{\bf Conflict of interest:} On behalf of all authors, the corresponding author states
that there is no conflict of interest.

\bibliography{ref-3}
\bibliographystyle{amsalpha}
\end{document}